\documentclass[12pt]{amsart}
\usepackage{etex}
\usepackage{amsfonts, amssymb, latexsym}

\usepackage{amscd,amssymb}
\usepackage{verbatim}
\usepackage{pstricks}
\usepackage{pst-node}
\usepackage[mathscr]{euscript}
\usepackage{tikz}
\usepackage{xypic}
\usetikzlibrary{matrix,arrows}
\usepackage[normalem]{ulem}
\usepackage{varwidth}
\usepackage{leftidx}

\newsymbol\pp 1275
\usepackage{tikz}
\usetikzlibrary{matrix,arrows,backgrounds,shapes.misc,shapes.geometric,patterns,calc,positioning}

\usepackage[colorlinks=true, pdfstartview=FitV, linkcolor=blue, citecolor=blue, urlcolor=blue]{hyperref}

\usepackage{fullpage}
\usepackage{graphicx}
\usepackage{enumerate}

\def\VR{\kern-\arraycolsep\strut\vrule &\kern-\arraycolsep}
\def\vr{\kern-\arraycolsep & \kern-\arraycolsep}

\newtheorem{theorem}{Theorem}

\newtheorem{lemma}[theorem]{Lemma}
\newtheorem{prop}[theorem]{Proposition}
\newtheorem{corollary}[theorem]{Corollary}

\theoremstyle{definition}
\newtheorem{definition}[theorem]{Definition}

\newtheorem*{prob}{Problem}

\newtheorem{rmk}[theorem]{Remark}
\newenvironment{remark}[1][]{\begin{rmk}[#1]\pushQED{\qed}}{\popQED \end{rmk}}

\newtheorem{qu}[theorem]{Question}

\newtheorem*{rmknonum}{Remark}

\newtheorem{obs}[theorem]{Observation}

\newtheorem{ex}[theorem]{Example}
\newenvironment{example}[1][]{\begin{ex}[#1]\pushQED{\qed}}{\popQED \end{ex}}
\newcommand{\pctext}[2]{\text{\parbox{#1}{\centering #2}}}

\newcommand{\Hom}{\operatorname{Hom}}
\newcommand{\End}{\operatorname{End}}
\newcommand{\Ext}{\operatorname{Ext}}

\newcommand{\Ann}{\operatorname{Ann}}
\newcommand{\ext}{\operatorname{ext}}
\newcommand{\rep}{\operatorname{rep}}

\newcommand{\SI}{\operatorname{SI}}

\newcommand{\GL}{\operatorname{GL}}

\newcommand{\ZZ}{\mathbb Z}
\newcommand{\CC}{\mathbb C}
\newcommand{\RR}{\mathbb R}
\newcommand{\NN}{\mathbb N}

\newcommand{\I}{\mathcal I}
\newcommand{\A}{\mathcal A}
\newcommand{\LL}{\mathcal L}

\newcommand{\V}{W}

\newcommand{\Ima}{\operatorname{Im}}

\newcommand{\supp}{\operatorname{supp}}

\newcommand{\ddim}{\operatorname{\mathbf{dim}}}

\newcommand{\dd}{\operatorname{\mathbf{d}}}

\newcommand{\pdim}{\mathsf{pdim}}
\newcommand{\C}{\mathcal{C}}

\newcommand{\p}{\mathcal{P}}
\newcommand{\Ar}{\mathcal{P}}
\newcommand{\s}{\mathcal{S}}
\newcommand{\K}{\mathcal{K}}
\newcommand{\jj}{\mathcal{I}^{-}_j}
\newcommand{\ii}{\mathcal{I}^{+}_i}
\newcommand{\Span}{\mathsf{Span}}
\newcommand{\capa}{\mathbf{cap}}
\newcommand{\Det}{\mathsf{Det}}

\newcommand{\R}{\operatorname{\mathcal{R}}}

\newcommand{\module}{\operatorname{mod}}
\newcommand{\codim}{\operatorname{codim}}

\newcommand\restr[2]{{
  \left.\kern-\nulldelimiterspace 
  #1 
  \vphantom{\big|} 
  \right|_{#2} 
  }}

\newcommand{\onto}{\twoheadrightarrow}

\begin{document}
\title{ Edmonds' problem and the membership problem for orbit semigroups of quiver representations}
\author{Calin Chindris}
\address{University of Missouri-Columbia, Mathematics Department, Columbia, MO, USA}
\email[Calin Chindris]{chindrisc@missouri.edu}

\author{Daniel Kline}
\address{College of the Ozarks, Mathematics Department, Point Lookout, MO, USA}
\email[Daniel Kline]{dkline@cofo.edu}

\date{\today}
\bibliographystyle{amsalpha}

\begin{abstract}  A central problem in algebraic complexity, posed by J. Edmonds \cite{Edm-1967}, asks to decide if the span of a given $l$-tuple $\V=(\V_1, \ldots, \V_l)$ of $N \times N$ complex matrices contains a non-singular matrix. 

In this paper, we provide a quiver invariant theoretic approach to this problem. Viewing $\V$ as a representation of the $l$-Kronecker quiver $\K_l$, Edmonds' problem can be rephrased as asking to decide if there exists a semi-invariant on the representation space $(\CC^{N\times N})^l$ of weight $(1,-1)$ that does not vanish at $\V$. In other words, Edmonds' problem is asking to decide if the weight $(1,-1)$ belongs to the orbit semigroup of $\V$.

Let $Q$ be an arbitrary acyclic quiver and $\V$ a representation of $Q$. We study the membership problem for the orbit semi-group of $\V$ by focusing on the so-called $\V$-saturated weights. We first show that for any given $\V$-saturated weight $\sigma$, checking if $\sigma$ belongs to the orbit semigroup of $\V$ can be done in deterministic polynomial time. 

Next, let $(Q, \R)$ be an acyclic bound quiver with bound quiver algebra $A=KQ/\langle \R \rangle$ and assume that $\V$ satisfies the relations in $\R$. We show that if $A/\Ann_A(\V)$ is a tame algebra then any weight $\sigma$ in the weight semigroup of $\V$ is $\V$-saturated. 

Our results provide a systematic way of producing families of tuples of matrices for which Edmonds' problem can be solved effectively.  
\end{abstract}

\maketitle
\setcounter{tocdepth}{1}
\tableofcontents

\section{Introduction}

\subsection{Motivation} \label{intro-sec} In \cite{Edm-1967}, Edmonds posed the following problem: Given an $l$-tuple of $N \times N$ matrices $\V=(\V_1, \dots, \V_l) \in (\CC^{N \times N})^l$, decide if $\Span_{\CC}(\V_1, \ldots, \V_l)$ contains a non-singular $N \times N$ matrix. The deterministic complexity of this problem plays a central role in algebraic complexity theory. In fact, according to \cite{Kab-Imp-2004}, the existence of a deterministic polynomial time algorithm for Edmonds' Problem implies non-trivial arithmetic circuit lower bounds. 

Our goal in this paper is to study Edmonds' problem within the framework of quiver invariant theory. To capture the classical Edmonds' problem, we begin by encoding the $l$-tuple of matrices as a representation $\V$ of the generalized $l$-Kronecker quiver, as follows:
$$\mathcal{K}_l:~
\vcenter
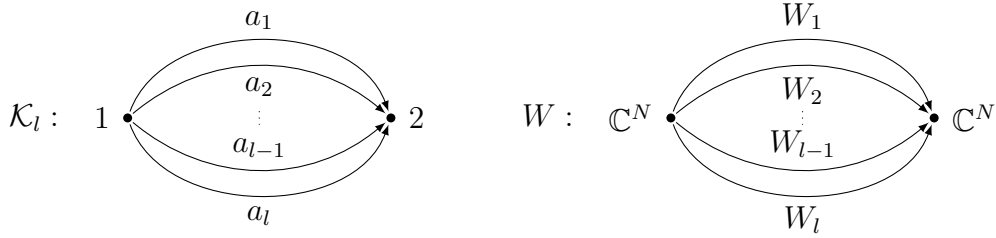
{\hbox{  
\begin{tikzpicture}[point/.style={shape=circle, fill=black, scale=.3pt,outer sep=3pt},>=latex]
   \node[point,label={left:$1$}] (1) at (0,0) {};
   \node[point,label={right:$2$}] (2) at (3.5,0) {};
  
   \draw[dotted] (1.75,.09)--(1.75,-.09);
  
   \path[->]
   (1) edge [bend left=70] node[midway, above] {$a_1$} (2)
   (1) edge [bend left=40] node[midway, below] {$a_2$} (2)
   (1) edge [bend right=40]  node[midway, above] {$a_{l-1}$} (2)
   (1) edge [bend right=70] node[midway, below] {$a_l$} (2);
\end{tikzpicture} 
}}
\hspace{30pt}
\V:~
\vcenter{\hbox{
\begin{tikzpicture}[point/.style={shape=circle, fill=black, scale=.3pt,outer sep=3pt},>=latex]
   \node[point,label={left:$\CC^N$}] (1) at (0,0) {};
   \node[point,label={right:$\CC^{N}$}] (2) at (3.5,0) {};
  
   \draw[dotted] (1.75,.09)--(1.75,-.09);
  
   \path[->]
   (1) edge [bend left=70] node[midway, above] {$\V_1$} (2)
   (1) edge [bend left=40] node[midway, below] {$\V_2$} (2)
   (1) edge [bend right=40]  node[midway, above] {$\V_{l-1}$} (2)
   (1) edge [bend right=70] node[midway, below] {$\V_l$} (2);
\end{tikzpicture} 
}}
$$

Let us next consider the \emph{orbit semigroup} of $\V$ defined by
\begin{alignat*}{2}
\s_{\mathcal{K}_l}(\V) & :=\Biggl\{ \sigma=(\sigma_1, \sigma_2) \in \ZZ^2 &&\;\Bigg|\; \pctext{3in}{$\exists$ a polynomial function $f$ on $(\CC^{N \times N})^l$ such that $f(\V) \neq 0$ and $g \cdot f = \det(g_1)^{\sigma_1}\det(g_2)^{\sigma_2}f$ for all $g=(g_1,g_2) \in \GL(N, \CC)\times \GL(N, \CC)$}   \Biggr\}, \\
\end{alignat*}
where the base change group $\GL(N, \CC)\times \GL(N, \CC)$ acts on $(\CC^{N \times N})^l$ by simultaneous conjugation. Using the First Fundamental Theorem (FFT) for quiver semi-invariants (see Section \ref{semi-inv-FFT-orbit-semigr-sec} for details), Edmonds' problem can be rephrased as asking to \emph{decide if the weight $(1,-1)$ belongs to the orbit semigroup $\s_{\mathcal{K}_l}(\V)$}. 

In this paper, we use the formalism of quiver invariant theory to study the membership problem for orbit semigroups of representations of arbitrary acyclic quivers. In particular, this approach provides a systematic way of constructing infinite families of large tuples of matrices for which Edmonds' problem can be solved effectively.

\subsection{Quiver formulation of Edmonds' problem} We briefly recall just enough terminology to state our quiver version of Edmonds' problem and the main results. (More detailed background can be found in Section \ref{background-sec}.) 

Let $Q$ be a connected acyclic quiver with set of vertices $Q_0$ and set of arrows $Q_1$. For an arrow $a \in Q_1$, we denote by $ta$ and $ha$, its tail and head, respectively. We represent $Q$ as a directed graph with set of vertices $Q_0$ and directed edges $a:ta \to ha$ for every $a \in Q_1$. A representation $\V$ of $Q$ assigns a finite-dimensional complex vector space $\V(x)$ to every vertex $x \in Q_0$ and a $\CC$-linear map $\V(a): \V(ta) \to \V(ha)$ to every arrow $a \in Q_1$. After fixing bases for the vector spaces $\V(x)$, $x \in Q_0$, we often think of the linear maps $\V(a)$, $a \in Q_1$, as matrices of appropriate size. The dimension vector of a representation $\V$ of $Q$ is $\ddim \V:=(\dim_{\CC} \V(x))_{x \in Q_0} \in \NN^{Q_0}$.

Let $\beta \in \NN^{Q_0}$ be a dimension vector and $\sigma \in \ZZ^{Q_0}$ a non-zero integral weight such that 
$$
\sigma \cdot \beta:=\sum_{x \in Q_0} \sigma(x) \beta(x)=0.
$$ 
Let $v_1, \ldots, v_n$ be the vertices of $Q$ where $\sigma$ takes positive values, and let $w_1, \ldots, w_m$ be the vertices of $Q$ where $\sigma$ takes negative values. Define
$$\sigma_{+}(v_i)=\sigma(v_i), \forall i \in [n], \text{~and~}
\sigma_{-}(w_j)=-\sigma(w_j), \forall j \in [m].
$$
For each $j \in [m]$ and $i \in [n]$, define 
$$
\jj:=\{q \in \ZZ \mid \sum_{k=1}^{j-1} \sigma_{-}(w_k) < q \leq  \sum_{k=1}^{j} \sigma_{-}(w_k) \},
$$
and 
$$
\ii:=\{r \in \ZZ \mid \sum_{k=1}^{i-1} \sigma_{+}(v_k) < r \leq  \sum_{k=1}^{i} \sigma_{+}(v_k) \}.
$$

Now let $\V$ be a $\beta$-dimensional representation of $Q$ with $\V(x)=\CC^{\beta(x)}$, $\forall x \in Q_0$, and $\V(a) \in \CC^{\beta(ha)\times \beta(ta)}$, $a \in Q_1$. To state our formulation of Edmonds' problem for the quiver datum $(\V, \sigma)$, we will work with large tuples of matrices indexed by the following set
$$
\I_{\sigma}=\{ (i,j,p,q,r)  \mid i \in [n], j \in [m], p \in \Ar_{i,j}, q \in \jj, r \in \ii \}.
$$
Here, $\Ar_{i,j}$ denotes the set of all oriented paths in $Q$ from $v_i$ to $w_j$ for all $i \in [m]$ and $j \in [m]$. 

Let $M:=\sum_{j=1}^m \sigma_{-}(w_j)$ and $M':=\sum_{i=1}^n \sigma_{+}(v_i)$. For every index $(i, j, p, q, r) \in \I_{\sigma}$, let $\V^{i,j,p}_{q,r}$ be the $M \times M'$ block matrix whose $(q,r)$-block-entry is $\V(p) \in \CC^{\beta(w_j)\times \beta(v_i)}$, and all other entries are zero matrices of appropriate size. All these block matrices are square matrices of size $N \times N$ where $N:=\sum_{i=1}^n \sigma_{+}(v_i)\beta(v_i)=\sum_{j=1}^m \sigma_{-}(w_j)\beta(w_j)$.

\begin{prob}[\textbf{Edmonds' problem for quiver data}] Given a quiver datum $(\V, \sigma)$ with $\V=(\V(a))_{a \in Q_1} \in \prod_{a \in Q_1} \CC^{\beta(ha)\times \beta(ta)}$, decide if $\Span_{\CC}(W^{i,j,p}_{q,r} \mid (i,j,p,q,r) \in \I_{\sigma})$ contains a $N \times N$ non-singular matrix.
\end{prob}

For example, if $\mathcal{K}_l$ is the $l$-Kronecker quiver, $\V=(\V_1, \ldots, \V_l) \in (\CC^{N \times N})^l$, and $\sigma=(1,-1)$ then the  matrices $\V^{i,j,p}_{q,r}$ are precisely the matrices $\V_1, \ldots, \V_l$. Thus Edmonds' problem for the quiver datum $(\V, (1,-1))$ over $\mathcal{K}_l$ is the classical formulation of Edmonds' problem.

Going back to the general situation, the orbit semigroup of $\V$, denoted by $\s_Q(\V)$, is the affine semigroup consisting of all weights $\sigma \in \ZZ^{Q_0}$ such that there exists a semi-invariant of weight $\sigma$ that does not vanish at $\V$. Then, via the FFT for quiver semi-invariants (see Section \ref{semi-inv-FFT-orbit-semigr-sec}), we have the following description 
$$
\s_Q(\V)=\left \{ \sigma \in \ZZ^{Q_0} \;\middle|\;
\begin{array}{l}
\sigma \cdot \beta=0 \text{~and~}\Span_{\CC}(W^{i,j,p}_{q,r} \mid (i,j,p,q,r) \in \I_{\sigma})\\
\text{~contains an~} N \times N \text{~non-singular matrix~} 
\end{array} \right \}
$$
Thus the \emph{membership problem for orbit semigroups} of quiver representations and \emph{Edmonds' problem} for quiver data are equivalent.

The following concept plays a key role in our approach. We say that a weight $\sigma \in \ZZ^{Q_0}$ is \emph{$\V$-saturated} if whenever $n \sigma \in \s_Q(\V)$ for some integer $n \geq 1$ then $\sigma$ also lies in $\s_Q(\V)$. Following L. Gurvits \cite{Gur-2004}, we say that the datum $(\V, \sigma)$ has the \emph{Edmonds-Rado Property (``ERP'')} if the tuple of matrices  $\A_{\V, \sigma}:=\left( \V^{i,j,p}_{q,r} : (i,j,p,q,r) \in \I_{\sigma} \right)$ has the Edmonds-Rado property, meaning that the existence of a non-singular matrix in $\Span_{\CC}(\A_{\V, \sigma})$ is equivalent to the capacity of the completely positive operator associated to $\A_{\V, \sigma}$ being positive. In that paper, Gurvits has found a deterministic polynomial time algorithm for testing the positivity of any completely positive operator, thus proving that Edmonds' problem can be solved effectively for any ERP tuple of matrices.

Our first result shows that for a quiver datum $(\V, \sigma)$, the weight $\sigma$ is $\V$-saturated if and only if $(\V, \sigma)$ is an ERP datum. In what follows, the capacity of an arbitrary quiver datum $(\V, \sigma)$ is denoted by  $\capa_{Q}(\V, \sigma)$ (see Definition \ref{defn-cap-quiver-datum}). The following result has been proved in \cite{ChiDer-2019} for bipartite quivers (over the field of real numbers). Here we explain how to extend it to arbitrary acyclic quivers.

\begin{theorem}[\textbf{Checking membership to orbit semigroups}] \label{membership-thm-1} Let $Q$ be a connected acyclic quiver and $(\V, \sigma)$ a quiver datum. Then the following statements are equivalent:
\begin{enumerate}[(i)]
\item $n \sigma \in \s_Q(\V)$ for some integer $n \geq 1$;
\item $ \capa_Q(\V, \sigma)>0$.
\end{enumerate} 

Consequently, $\sigma$ is $\V$-saturated if and only if $(\V, \sigma)$ is an ERP datum, and hence Algorithm G in \cite[Corollary 3.17]{GarGurOliWig-2018} yields a deterministic polynomial time algorithm to check if $\sigma \in \s_Q(\V)$ for any $\V$-saturated weight $\sigma$.
\end{theorem}

It is thus  important to have a systematic way of constructing saturated weights (equivalently, ERP quiver data). Working with the concept of saturated weights allows us to bring methods from the invariant theory for finite-dimensional algebras to bear in the context of the Edmonds' problem. For a representation $\V$ of a bound quiver $(Q, \R)$, $\LL_{\V}$ denotes its weight semigroup (for more details, see Definition \ref{weight-semigrp-defn}).

\begin{theorem} \label{main-thm-2} Let $(Q, \R)$ be a connected, acyclic bound quiver with bound quiver algebra $A=\CC Q/ \langle \R \rangle$. Let $\beta \in \NN^{Q_0}$ be a dimension vector and $\V \in \module(A,\beta)$ an $A$-module such that $B:=A/ \Ann_A(\V)$ is a tame algebra. 

Then every weight $\sigma \in \LL_{\V}$ is $\V$-saturated, and hence Edmonds' problem for $(\V,\sigma)$ can be solved in deterministic polynomial time.
\end{theorem}

Since any quotient of a tame algebra is tame, as a particular case of Theorem \ref{main-thm-2}, we obtain that for any tame algebra $A$ and $A$-module $\V$, every weight in $\LL_{\V}$ is $\V$-saturated. Going beyond the class of tame algebras, we point out that there exist wild (but Schur-tame) algebras $A$ such that every weight in the weight semigroup of a Schur module is saturated. In \cite{Gur-2004}, Gurvits has found two classes of matrix spaces that have the Edmonds-Rado property (see also \cite{IKQS15}). Namely, the class $\mathbf{R_1}$ of rank-1 spanned matrix spaces, and the class $\mathbf{UT}$ of (upper) triangularizable matrix spaces. The ERP tuples of matrices arising from Theorem \ref{main-thm-2} and \cite[Theorem 1]{Chi-orbit-semigr-2009} go beyond the two classes $\mathbf{R1}$ and $\mathbf{UT}$. 

The proof of Theorem \ref{main-thm-2} is based on the FFT for quiver semi-invariants, and the generic decomposition of irreducible components of module varieties due to V. Kac \cite{Kac}, and W. Crawley-Boevey and J. Schr{\"o}er \cite{C-BS}. More specifically, we are interested in irreducible components of the following form. Let $(Q, \R)$ be an acyclic bound quiver with bound quiver algebra $A=\CC Q/\langle \R \rangle$. For a dimension vector $\alpha \in \NN^{Q_0}$ for which there exists an $A$-module of projective dimension at most one, $\C(\alpha)$ denotes the unique irreducible component of $\module(A, \alpha)$ whose generic module has projective dimension at most one (see \cite{GeiSch}). The following result, which is interesting in itself, plays a key role in the proof of Theorem \ref{main-thm-2}.

\begin{prop} \label{main-prop} Let $A=\CC Q/\langle \R \rangle$ be a connected, acyclic, tame bound quiver algebra and let $\C(\alpha) \subseteq \module(A, \alpha)$ be an irreducible component whose generic module has projective dimension at most one.
\begin{enumerate}[(i)]
\item If $\C(\alpha)$ is an indecomposable irreducible component then
$$
\ext_A^1(\C(\alpha), \C(\alpha))=0.
$$

\item Let
$$
\C(\alpha)=\overline{\C(\alpha_1) \oplus \ldots \oplus \C(\alpha_l)}
$$
be the generic decomposition of $\C(\alpha)$. Then, for any integer $n>0$, the generic decomposition of $\C(n \alpha)$ is
$$
\C(n\alpha)=\overline{\C(\alpha_1)^{\oplus n} \oplus \ldots \oplus \C(\alpha_l)^{\oplus n}}.
$$

\item The dimension of $\C(\alpha)$ is equal to the dimension of $\GL(\alpha)$ if and only if none of the indecomposable irreducible components that occur in the generic decomposition of $\C(\alpha)$ is an orbit closure.
\end{enumerate}
\end{prop}
\noindent

We point out that Proposition \ref{main-prop}{(ii)} can be viewed as a generalization of \cite[Theorem 3.8]{S1} while part (iii) can be viewed as a generalization of \cite[Propositon 4.3]{CarChiKinWey2017} and \cite[Theorem 1.6]{GeiLabSch-2020} to arbitrary acyclic tame algebras.

\bigskip
\noindent
\textbf{Acknowledgements:} The authors would like to thank Ryan Kinser for helpful discussions on the subject and for bringing \cite{GeiLabSch-2020} to their attention.

\section{Background on quiver invariant theory} \label{background-sec}
\subsection{Representations of bound quivers}
Throughout, we work over the field $K=\CC$ of complex numbers and denote by $\NN=\{0,1,\dots \}$. For a positive integer $L$, we denote by $[L]=\{1, \ldots, L\}$. All algebras are assumed to be bound quiver algebras, and all modules are assumed to be finite-dimensional left modules.

A \emph{quiver} $Q=(Q_0,Q_1,t,h)$ consists of two finite sets $Q_0$ (\emph{vertices}) and $Q_1$ (\emph{arrows}) together with two maps $t:Q_1 \to Q_0$ (\emph{tail}) and $h:Q_1 \to Q_0$ (\emph{head}). We represent $Q$ as a directed graph with set of vertices $Q_0$ and directed edges $a:ta \to ha$ for every $a \in Q_1$. Throughout we assume that our quivers are connected, meaning that the underlying graph of $Q$ is connected.

A \emph{representation} of $Q$ is a family $V=(V(x), V(a))_{x \in Q_0, a\in Q_1}$ where $V(x)$ is a finite-dimensional $K$-vector space for every $x \in Q_0$, and $V(a): V(ta) \to V(ha)$ is a $K$-linear map for every $a \in Q_1$. A \emph{subrepresentation} $V'$ of $V$, written as $V' \leq V$, is a representation of $Q$ such that $V'(x) \leq_{K} V(x)$ for every $x \in Q_0$, and $V(a)(V'(ta)) \leq V'(ha)$ and $V'(a)$ is the restriction of $V(a)$ to $V(ta)$ for every arrow $a \in Q_1$. 

A morphism $\varphi:V \rightarrow W$ between two representations is a collection $(\varphi(x))_{x \in Q_0}$ of $K$-linear maps with $\varphi(x) \in \Hom_K(V(x), \V(x))$ for each $x \in Q_0$, and such that $\varphi(ha) \circ V(a)=\V(a) \circ \varphi(ta)$ for each $a \in Q_1$. We denote by $\Hom_Q(V,\V)$ the $\CC$-vector space of all morphisms from $V$ to $\V$. The abelian category of all representations of $Q$ is denoted by $\rep(Q)$.

The path algebra $KQ$ of a quiver $Q$ has a $K$-basis consisting of all paths (including the trivial ones), and the multiplication in $KQ$ is given by concatenation of paths. It is easy to see that any $KQ$-module defines a representation of $Q$, and vice-versa. Furthermore, the category $\module(KQ)$ of finite-dimensional $KQ$-modules is equivalent to the category $\rep(Q)$. In fact, we identify $\module(KQ)$ and $\rep(Q)$, and use the same notation for a module and the corresponding representation.

An admissible relation for $Q$ is a finite linear combination of parallel paths where each path has length at least two. A \emph{bound quiver} is a pair $(Q, \R)$ where $Q$ is a quiver and $\R$ is a finite set of admissible relations such that the algebra $KQ/\langle \R \rangle$ is finite-dimensional. 

Let $(Q,\R)$ be a bound quiver and $A=KQ/\langle \R \rangle$ its bound quiver algebra. A representation $V$ of $A$ (or  $(Q,\R)$) is just a representation $V$ of $Q$ such that $V(r)=0$ for all $r \in \R$. The category $\module(A)$ of finite-dimensional left $A$-modules is equivalent to the category $\rep(A)$ of representations of $A$. As before, we identify $\module(A)$ and $\rep(A)$, and make no distinction between $A$-modules and representations of $A$.  For each vertex $x \in Q_0$, we denote by $P_x$ the projective indecomposable cover of the simple $A$-module $S_x$. An $A$-module $V$ is called \emph{Schur} if $\End_A(V) \cong K$. 

By a dimension vector of $A$ (equivalently, of $Q$), we simply mean a $\NN$-valued function on the set of vertices $Q_0$. For two vectors $\theta, \beta \in \RR^{Q_0}$, we define $\theta \cdot \beta=\sum_{x \in Q_0} \theta(x)\beta(x)$. 

From now on we assume that $Q$ has no oriented cycles. Then the \emph{Euler form} of $A$ is the bilinear form  $\langle \cdot, \cdot \rangle_{A} : \ZZ^{Q_0}\times \ZZ^{Q_0} \to \ZZ$ defined by
$$
\langle \alpha,\beta  \rangle_{A}=\sum_{l\geq 0}(-1)^l \sum_{x,y\in Q_0}\dim_K \Ext^l_{A}(S_x,S_y)\alpha(x)\beta(y).
$$
In fact, for any $A$-modules $V$ and $\V$ of dimension vector $\alpha$ and
$\beta$, respectively, one has
$$
\langle \alpha, \beta \rangle_{A}=\sum_{l\geq 0}(-1)^l \dim_K \Ext^l_{A}(V,\V).
$$

\subsection{Module varieties and their irreducible components}\label{repvar-sec} Let $\alpha \in \NN^{Q_0}$ be a dimension vector of $A=KQ/\langle \R \rangle$ (or equivalently, of $Q$). The \emph{representation space} of $\alpha$-dimensional representations of $Q$ is the affine space 
$$\rep(Q,\alpha):=\prod_{a \in Q_1} K^{\alpha(ha)\times \alpha(ta)}.$$ The \emph{module variety} of $\alpha$-dimensional $A$-modules is 
$$\module(A,\alpha):=\{(V(a))_{a \in Q_1} \in \rep(Q,\dd) \mid V(r)=\mathbf{0}, \forall r \in \R\}.$$ It is acted upon by the base change group $$\GL(\alpha):=\prod_{x\in Q_0}\GL(\alpha(x),K)$$ by simultaneous conjugation, i.e., for $g=(g(x))_{x\in Q_0}\in \GL(\alpha)$ and $V=(V(a))_{a \in Q_1} \in \module(Q,\alpha)$,  $g \cdot V$ is defined by 
$$
(g\cdot V)(a)=g(ha)V(a) g(ta)^{-1}, \forall a \in Q_1.
$$ 
The $\GL(\alpha)-$orbits in $\module(A,\alpha)$ are in one-to-one correspondence with the isomorphism classes of the $\alpha$-dimensional $A$-modules. 

In general, $\module(A, \alpha)$ does not have to be irreducible. An irreducible component $C \subseteq \module(A, \alpha)$ is said to be \emph{indecomposable} if $C$ has a non-empty open subset of indecomposable modules. Given a decomposition $\alpha=\alpha_1+\ldots +\alpha_l$ where $\alpha_i \in \NN^{Q_0}, 1 \leq i \leq l$, and $\GL(\alpha_i)$-invariant constructible subsets $C_i\subseteq \module(A,\alpha_i)$, $1 \leq i \leq l$, we denote by $C_1\oplus \ldots \oplus C_l$ the constructible subset of $\module(A,\alpha)$ defined by 
$$
C_1\oplus \ldots \oplus C_l=\{V \in \module(A,\alpha) \mid V\simeq \bigoplus_{i=1}^t V_i\text{~with~} V_i \in C_i, \forall 1 \leq i \leq l\}.
$$

It is proved in \cite[Theorems~1.1 and 1.2]{C-BS} that any irreducible component of a module variety has a Krull-Schmidt type decomposition. Before stating this important result, let us recall that the \emph{generic Ext} between two irreducible components $D$ and $E$ is defined as $\ext_A^1(D,E):=\min \{\dim_K \Ext^1_A(X,Y) \mid (X,Y) \in D \times E\}$.

\begin{theorem} \label{decomp-irr-thm} 
\begin{enumerate}
\item If $C$ is an irreducible component of $\module(A,\alpha)$ then there exist dimension vectors $\alpha_1, \ldots, \alpha_l$ of $A$ such that $\alpha=\alpha_1+\ldots +\alpha_l$ and
$$
C=\overline{C_1\oplus \ldots \oplus C_l}
$$
for some indecomposable irreducible components $C_i\subseteq \module(A,\alpha_i), 1 \leq i \leq l$. Moreover, the indecomposable irreducible components $C_i, 1 \leq i \leq l,$ are uniquely determined by this property, up to reordering. The decomposition $C=\overline{C_1\oplus \ldots \oplus C_l}$ is called the \emph{generic decomposition of $C$}.

\item Conversely, if $C_i \subseteq \module(A,\alpha_i)$, $1 \leq i \leq l$, are indecomposable irreducible components then $\overline{C_1\oplus \ldots \oplus C_l}$ is an irreducible component of $\module(A,\sum_{i=1}^l \alpha_i)$ if and only if $\ext_A^1(C_i,C_j)=0$ for all $1 \leq i \neq j \leq l$.
\end{enumerate}
\end{theorem}

We will be particularly interested in irreducible components of the following form. For an $A$-module $V$, we denote by $\pdim V$ the \emph{projective dimension} of $V$. For a dimension vector $\alpha \in \NN^{Q_0}$, consider the following (possibly empty) set 
$$
\p_A(\alpha):=\{ V \in \module(A,\alpha) \mid \pdim_A V \leq 1\} \subseteq \module(A, \alpha),
$$
and set $\C(\alpha):=\overline{\p_A(\alpha)}$. 

\begin{prop} \label{irr-comp-proj-1-prop} (see \cite{GeiSch}) Let $\alpha \in \NN^{Q_0}$ be a dimension vector such that $\p_A(\alpha) \neq \emptyset$. Then $\p_A(\alpha)$ is an irreducible open subset of $\module(A, \alpha)$, and thus $\C(\alpha)$ is an irreducible component of $\module(A, \alpha)$.
\end{prop}

It is immediate to see that if $\p_A(\alpha) \neq \emptyset$ then the generic decomposition of $\C(\alpha)$ is of the form
$$
\C(\alpha)=\overline{\C(\alpha_1)\oplus \ldots \oplus \C(\alpha_l)}
$$
with $\C(\alpha_1), \ldots, \C(\alpha_l)$ indecomposable irreducible components such that $\ext^1_A(\C(\alpha_i), \C(\alpha_j))=0$ for all $1 \leq i \neq j \leq l$.

\subsection{Semi-invariants and orbit semigroups} \label{semi-inv-FFT-orbit-semigr-sec}
Let $\beta \in \NN^{Q_0}$ be a dimension vector of $Q$ and consider the action of the base change group $\GL(\beta)=\prod_{x \in Q_0}\GL(\beta(x), K)$ on the representation space $\rep(Q,\beta)=\prod_{a \in Q_1} K^{\beta(ha)\times \beta(ta)}$. 

\begin{definition} Let $\sigma \in \ZZ^{Q_0}$ be an integral weight of $Q$ such that $\sigma \cdot \beta =0$. 

\begin{enumerate}
\item The \emph{character} of $\GL(\beta)$ \emph{induced by $\sigma$} is 

\begin{align*}
\chi_{\sigma}:\GL(\beta) & \to K \setminus \{0\} \\
g=(g(x))_{x \in Q_0} & \mapsto \chi_{\sigma}(g)=\prod_{x \in Q_0} \det(g(x))^{\sigma(x)}
\end{align*}

\item A polynomial function $F \in K[\rep(Q,\beta)]$ is called a \emph{semi-invariant of weight $\sigma$} on $\rep(Q, \beta)$ if
$$
g \cdot F=\chi_{\sigma}(g)F, \forall g \in \GL(\beta).
$$
The space of all semi-invariants on $\rep(Q,\beta)$ of weight $\sigma$ is denoted by $\SI(Q,\beta)_{\sigma}$. \\

\item More generally, let us assume that $(Q, \R)$ is a bound quiver. A regular function $f \in K[\module(A, \beta)]$ is called a \emph{semi-invariant of weight} $\sigma$ on $\module(A, \beta)$ 
if $g \cdot f=\chi_{\sigma}(g) f$ for all $g \in \GL(\beta)$. We denote by $\SI(A, \beta)_{\sigma}$ the space of semi-invariants of weight $\sigma$ on $\module(A, \beta)$.
\end{enumerate}
\end{definition}

\begin{remark} \label{semi-inv-relns-rmk} Since $\GL(\beta)$ is \emph{linearly} reductive in characteristic zero and $\module(A, \beta)$ is a $\GL(\beta)$-invariant closed subvariety of $\rep(Q, \beta)$, we have that
\[
\SI(A, \beta)_{\sigma}=\{\restr{F}{\module(A, \beta)} \mid F \in \SI(Q, \beta)_{\sigma}\}.
\]
\end{remark}

The following remarkable result has been proved by Derksen-Weyman \cite{DW1}, Domokos-Zubkov \cite{DZ}, and Schofield-van den Bergh \cite{SVB}.

\begin{theorem}[\textbf{FFT for quiver semi-invariants}] (see for example \cite[Corollary 3]{DW1}) \label{FFT-det-thm} Keep the same notation as above. Then the coefficients of the polynomial
$$
\det\left( \sum_{(i,j,p,q,r)} t^{i,j,p}_{q,r} \V^{i,j,p}_{q,r} \right) \in K[\rep(Q,\beta)][t^{i,j,p}_{q,r}: (i,j,p,q,r) \in \I_{\sigma}]
$$
span the weight space of semi-invariants $\SI(Q,\beta)_{\sigma}$. Here, $t^{i,j,p}_{q,r}$, $(i,j,p,q,r) \in \I_{\sigma}$, are indeterminate variables.
\end{theorem}

\begin{definition} Let $\V \in \rep(Q,\beta)$ be a $\beta$-dimensional representation. The \emph{orbit semigroup} of $\V$ is defined by
$$
\s_Q(\V):=\{\sigma \in \ZZ^{Q_0} \mid \exists f \in \SI(Q,\beta)_{\sigma} \text{~such that~} f(\V) \neq 0  \}.
$$
(When no confusion arises, we drop the subscript $Q$.)
\end{definition}

\begin{remark} \label{semi-inv-relns-rmk-2} Let $(Q, \R)$ be a bound quiver with bound quiver algebra $A=KQ/ \langle \R \rangle$ and $\beta \in \NN^{Q_0}$ a dimension vector. For an $A$-module $\V \in \module(A, \beta)$, we define
$$
\s_A(\V):=\{\sigma \in \ZZ^{Q_0} \mid \exists f \in \SI(A,\beta)_{\sigma} \text{~such that~} f(\V)\neq 0 \}.
$$
Using Remark \ref{semi-inv-relns-rmk}, we have that
$$
\s_A(\V)=\s_Q(\V).
$$

One of the advantages of working with the algebra $A$ and the affine variety $\module(A, \beta)$ instead of just the hereditary path algebra $KQ$ and the affine space $\rep(Q, \beta)$ is that there are very many cases where $Q$ is wild while $A$ is tame.
\end{remark}

As a direct consequence of Theorem \ref{FFT-det-thm}, we obtain the following description of the orbit semigroup of a quiver representation.

\begin{corollary}\label{membership-Edmonds-coro} Let $\V \in \rep(Q, \beta)$ be a $\beta$-dimensional representation of $Q$. Then
\begin{alignat*}{2}
\s_Q(\V) & :=\Biggl\{ \sigma \in \ZZ^{Q_0} &&\;\Bigg|\; \pctext{3.25in}{$\sigma \cdot \beta=0$ and $\Span_{\CC}(W^{i,j,p}_{q,r} \mid (i,j,p,q,r) \in \I_{\sigma})$ contains an $N \times N$ non-singular matrix}   \Biggr\}. \\
\end{alignat*}
In other words, the membership problem for $\s_Q(\V)$ and Edmonds' problem for $(\V, \sigma)$ are equivalent.
\end{corollary}

We will also need the following homological description of $\s_Q(\V)$. It is a consequence of the FFT for semi-invariants of bound quivers. We include a proof for the convenience of the reader.

\begin{prop} \label{orbit-semigrp-homological-prop} Let $(Q, \R)$ be a connected, acyclic bound quiver with bound quiver algebra $A=KQ/\langle \R \rangle$. Let $\beta \in \NN^{Q_0}$ be a dimension vector and $\V \in \module(A, \beta)$ an $A$-module. Let $B:=KQ/\Ann_Q(\V)$ and denote by $\langle \cdot, \cdot \rangle_B$ the Euler bilinear form of $B$. Then
$$
\s_Q(\V)=\left \{ \sigma \in \ZZ^{Q_0} \;\middle|\; 
\begin{array}{l}
\restr{\sigma}{\supp(\beta)}=\langle \alpha, - \rangle_{B} \text{~where~}  \alpha \in \NN^{\supp(\beta)} \text{~with~} \p_{B}(\alpha) \neq \emptyset   \\
\text{and such that there exists a~} B\text{-module~} V \in \p_{B}(\alpha) \text{~with} \\
\Hom_{B}(V, \V)=\Ext_{B}^1(V, \V)=0
\end{array}
\right \}
$$
\end{prop}

\begin{proof} Since $Q$ is acyclic, the Gabriel quiver of the algebra $B$ is a subquiver of $Q$ whose set of vertices is $\supp(\beta)=\{x \in Q_0 \mid \beta(x)\neq 0\}$. Furthermore, using Remark \ref{semi-inv-relns-rmk-2}, we get that
$$
\s_Q(\V)=\{\sigma \in \ZZ^{Q_0} \mid \restr{\sigma}{\supp(\beta)} \in \s_B(\V)\}.
$$

Thus, to prove our claim, we can simply assume without any loss of generality that $\V$ is a \emph{faithful} $A$-module. Then $A=B$, and the claim of the proposition can be rephrased as asking to show that
\begin{equation} \label{eqn-1}
\s_Q(\V)=\left \{ \sigma \in \ZZ^{Q_0} \;\middle|\; 
\begin{array}{l}
\sigma=\langle \alpha, - \rangle_{A} \text{~where~}  \alpha \in \NN^{Q_0} \text{~with~} \p_{A}(\alpha) \neq \emptyset   \\
\text{and such that there exists an~} A\text{-module~} V \in \p_{A}(\alpha) \text{~with} \\
\Hom_{A}(V, \V)=\Ext_{A}^1(V, \V)=0
\end{array}
\right \}
\end{equation}

Let $\sigma \in \ZZ^{Q_0}$ be a weight of the form $\sigma=\langle \alpha, - \rangle_{A}$ where $\alpha \in \NN^{Q_0}$ with $\p_{A}(\alpha) \neq \emptyset$ and such that there exists an $A$-module $V \in \p_{A}(\alpha)$ with
$\Hom_{A}(V, \V)=\Ext_{A}^1(V, \V)=0$. Then the Schofield's semi-invariant $c^{V} \in \SI(A, \beta)_{\sigma}$ has the property that $c^V(\V)\neq 0$ (see \cite{Domo}), and thus $\sigma \in \s_Q(\V)$. This proves the inclusion $``\supseteq"$ of $(\ref{eqn-1})$.

To prove the other inclusion, we know from the FFT for quiver semi-invariants (see \cite{DW4} or \cite{Domo}) that there exists a Schofield semi-invariant $c^V \in \SI(A, \beta)_{\sigma}$ with $V \in \module(A, \alpha)$ for some dimension vector $\alpha \in \NN^{Q_0}$ such that
\begin{enumerate}[(i)]
\item $\sigma(x)=\dim_K \Hom_A(V, S_x)-\dim_K \Hom(S_x, \tau_A V)$ for every $x \in Q_0$;

\item $c^V(\V) \neq 0 \Longleftrightarrow \Hom_A(V, \V)=\Hom_A(\V, \tau_A V)=0$.
\end{enumerate}

To complete the proof, we will show that $\pdim_A V \leq 1$. For this last step, we will need the assumption that $\V$ is a faithful $A$-module. In what follows, we are going to use the same arguments as in the proof of \cite[Theorem 1]{DW4}. Assume for a contradiction that $\pdim_A V \geq 2$ and let  
$$
\cdots \to P_2 {\overset{g} \longrightarrow} P_1 {\overset{f} \longrightarrow} P_0 
$$
be a minimal projective resolution of $V$. Let us consider the induced complex
$$
\Hom_A(P_0, \V) {\overset{\overline{f}} \longrightarrow} \Hom_A(P_1, \V) {\overset{\overline{g}} \longrightarrow} \Hom_A(P_2, \V).
$$
Recall that by definition $c^V(\V)=\det (\overline{f})$ and so $\overline{f}$ is an isomorphism of vector spaces by (ii). But his implies that $\Hom_A(P_1, \V)=\Ima(\overline{f}) \subseteq \ker(\overline{g})$, i.e. $\overline{g}$ is the zero map.

On the other hand, since $g$ is a homorphism between projective modules, we can view $g$ as a matrix whose entries are linear combinations of residue classes of parallel paths in $Q$. Now, let $p_{x,y}$ be a linear combination of oriented paths from $x \in Q_0$ to $y \in Q_0$ such that its residue class modulo $\langle \R \rangle$ is the $(x,y)$-entry of $g$. Then the $(x,y)$-entry of $\overline{g}$ is precisely the linear map $\V(p_{x,y})$. Since $\overline{g}$ is zero and $\V$ is a faithful $A$-module, we get that all the entries $p_{x,y}$ belong to $\langle \R \rangle$. This shows that $g$ is the zero homorphism, making $f$ injective (contradiction). So, we conclude that $\pdim V \leq 1$. This now implies that
$$
\sigma(x)=\langle \alpha, \ddim S_x \rangle_A, \forall x \in Q_0, \text{~i.e.~} \sigma=\langle \alpha, \cdot \rangle_A, 
$$
which finishes the proof.
\end{proof}

\begin{definition} \label{weight-semigrp-defn} Let $(Q, \R)$ be a connected, acyclic bound quiver with bound quiver algebra $A=KQ/\langle \R \rangle$. We define the the \emph{weight semigroup} of a module $\V \in \module(A,\beta)$ to be the semigroup
$$
\LL_{\V}=\left \{ \sigma \in \ZZ^{Q_0} \;\middle|\;  \restr{\sigma}{\supp(\beta)}=\langle \alpha, - \rangle_{B} \text{~where~}  \alpha \in \NN^{\supp(\beta)} \text{~with~} \p_{B}(\alpha) \neq \emptyset    \right \},
$$
where $B=A/\Ann_A(W) \simeq KQ/\Ann_{KQ}(W)$.
\end{definition}

\begin{rmk} According to Proposition \ref{orbit-semigrp-homological-prop}, a necessary condition for a weight $\sigma \in \ZZ^{Q_0}$ to belong to $\s_Q(\V)$ is that $\sigma$ belongs to $\LL_{\V}$, i.e.
$$
\s_Q(\V) \subseteq \LL_{\V} \subseteq \ZZ^{Q_0}.
$$
\end{rmk}

\section{Saturated weights and the capacity of quiver data}

\subsection{Quiver semi-stability and capacity of quiver data} Let $Q=(Q_0,Q_1, t,h)$ be a connected acyclic quiver, $\beta \in \ZZ_{>0}^{Q_0}$ a dimension vector, and $\sigma \in \ZZ^{Q}$ a non-zero weight of $Q$ such that $\sigma \cdot \beta=0$. Let $v_1, \ldots, v_n$ be the vertices of $Q$ where $\sigma$ takes positive values, and let $w_1, \ldots, w_m$ be the vertices of $Q$ where $\sigma$ takes negative values. Recall that $N=\sum_{i=1}^n \sigma_{+}(v_i)\beta(v_i)=\sum_{j=1}^m \sigma_{-}(w_j)\beta(w_j)$. 

Let $Q^{\sigma}$ be the bipartite quiver with partite sets $\{v_1, \ldots, v_n\}$ and $\{w_1, \ldots, w_m\}$, respectively. Furthermore, for every oriented path $p$ in $Q$ from $v_i$ to $w_j$, we draw an arrow $a_p$ from $v_i$ to $w_j$ in $Q^{\sigma}$. The restriction of $\beta$ (or $\sigma$) to $Q^{\sigma}$ is denoted by the same symbol. Note that the quiver $Q^{\sigma}$ does not change if $\sigma$ is replaced by any integer multiple $l\sigma$ with $l>0$.

We have the following important consequence of the FFT for quivers semi-invariants.

\begin{prop} \label{blow-ups-semi-inv-prop} Keep the same notation as above. Let $\varphi: \rep(Q, \beta) \to \rep(Q^{\sigma}, \beta)$ be the morphism defined by $\varphi(\V)(a_p)=\V(p)$ for all $\V \in \rep(Q,\beta)$ and arrows $a_p \in Q^{\sigma}_1$. Then $\varphi$ induces a surjective linear map $\varphi^{\#}:\SI(Q^{\sigma}, \beta)_{l \sigma} \onto \SI(Q,\beta)_{l \sigma}$ for all integers $l>0$.
\end{prop}

\begin{proof} Fix an integer $l \geq 1$ and denote $l \sigma$ by $\theta$. Let us now recall the FFT in terms of the so-called Schofield' semi-invariants. Let $f$ be an $m \times n$ block matrix whose $(j,i)$-block entry is a $\theta_{-}(w_j)\times \theta_{+}(v_i)$ matrix whose entries are $K$-linear combinations of oriented paths from $v_i$ to $w_j$ in $Q$. We point out that $f$ can simply be viewed as an element of $\Hom_Q(P_{-}, P_{+})$ where 
$$P_{-}=\bigoplus_{j=1}^m P_{w_j}^{\theta_{-}(w_j)} \text{~and~} P_{+}=\bigoplus_{i=1}^n P_{v_i}^{\theta_{+}(v_i)}.$$

For every $\V \in \rep(Q, \beta)$, we define $\V^f$ to be the $N \times N$ matrix obtained from $f$ by replacing every oriented path $p$ from $v_i$ to $w_j$ by the matrix $\V(p) \in K^{\beta(w_j) \times \beta(v_i)}$. (Note that $\V^f$ is precisely the linear map $\Hom_Q(f, \V): \Hom_Q(P_{+}, \V) \to \Hom_Q(P_{-}, \V)$ viewed as a matrix.) Thus we can define the regular map $c^f:\rep(Q,\beta) \to K$ by $c^f(\V)=\det(\V^f)$ for all $\V \in \rep(Q, \beta)$. It turns out that $c^f$, also known as the Scofield semi-invariant associated to $f$, is a semi-invariant on $\rep(Q, \beta)$ of weight $\theta$. Then the FFT for quiver semi-invariants (see \cite{DZ} or \cite{DW1})  can be stated as follows

\begin{equation} \label{FFT-eqn-1}
\SI(Q, \beta)_{\theta}=\langle c^f \mid f \in \Hom_Q(P_{-}, P_{+}) \rangle
\end{equation}

Now, if we denote by $\widetilde{P}_{-}$ and $\widetilde{P}_{+}$ the analogs of $P_{-}$ and $P_{+}$ for $Q^{\sigma}$, we can think of any $f \in \Hom_Q(P_{-}, P_{+})$, viewed as a matrix as above, as an element $\widetilde{f}$ of $\Hom_{Q^{\sigma}}(\widetilde{P}_{-}, \widetilde{P}_{+})$ by simply replacing an oriented path $p$ in $Q$ from $v_i$ to $w_j$ by the corresponding arrow $a_p$ in $Q^{\sigma}$. Moreover, for any $c^f \in \SI(Q, \beta)_{\theta}$ with $f \in \Hom_Q(P_{-}, P_{+})$, we have that
$$
c^f(\V)=c^{\widetilde{f}}(\varphi(\V)), \forall \V \in \rep(Q, \beta),
$$
i.e. 
$$c^f=\varphi^{\#}(c^{\widetilde{f}}).$$
The surjectivity of $\varphi^{\#}$ now follows from $(\ref{FFT-eqn-1})$.
\end{proof}

Let us now recall the definition of the capacity of a quiver datum $(\V, \sigma)$.

\begin{definition} (see \cite{ChiDer-2019}) \label{defn-cap-quiver-datum} Let $\V \in \rep(Q,\beta)$ be a representation of $Q$. 

\begin{enumerate}
\item The \emph{Brascamp-Lieb operator} associated to the quiver datum $(\V, \sigma)$ is the completely positive operator $T_{\V, \sigma}$ with Kraus operators $\{\V^{i,j,p}_{q,r} \mid (i,j,p,q,r) \in \I_{\sigma}\}$, i.e.
\begin{align*}
T_{\V,\sigma}: \CC^{N \times N} & \to \CC^{N \times N}\\
X& \to T_{\V,\sigma}(X):=\sum_{(i,j,p,q,r)}(\V^{i,j,p}_{q,r})^T\cdot X \cdot \V^{i,j,a}_{q,r}
\end{align*}

\item The \emph{capacity} of $(\V,\sigma)$, denoted by $\capa_Q(\V,\sigma)$,  is defined to be the capacity of $T_{\V, \sigma}$, i.e.
$$
\capa_Q(\V,\sigma):=\inf \{\Det(T_{\V,\sigma}(X)) \mid  X \in \s^{+}_N,  \Det(X)=1 \}.
$$
(Here, for a given positive integer $d$, we denote by $\s^{+}_d$ the set of all $d \times d$ (symmetric) positive definite real matrices.)
\end{enumerate}
\end{definition}

We are now ready to prove Theorem \ref{membership-thm-1}

\begin{proof}[Proof of Theorems \ref{membership-thm-1}] Let $\V^{\sigma}$ be the representation of $Q^{\sigma}$ defined by
\begin{itemize}
\item $\V^{\sigma}(v_i)=\V(v_i)$, $\V^{\sigma}(w_j)=\V(w_j)$ for all $i \in [n]$, $j \in [m]$, and
\smallskip
\item $\V^{\sigma}(a_p)=\V(p)$ for every arrow $a_p$ in $Q^{\sigma}$.
\end{itemize}	
It has been proved in \cite[Theorem 1]{ChiDer-2019} that
$$
\capa_Q(\V, \sigma)>0 \Longleftrightarrow \V^{\sigma} \text{~is~} \sigma-\text{semi-stable}.
$$
(Recall that a representation $M \in \rep(Q, \dd)$ is said to be \emph{$\sigma$-semi-stable} if $n \sigma \in \s_Q(M)$ for some positive integer $n \geq 1$.) Thus, to prove the equivalence $(i) \Longleftrightarrow (ii)$ it comes down to proving that
$$
\V \text{~is~} \sigma-\text{semi-stable} \Longleftrightarrow \V^{\sigma} \text{~is~} \sigma-\text{semi-stable as a representation of~} Q^{\sigma}.
$$ 

It is not difficult to prove the implication from right to left. Let us now check that if $\V$ is $\sigma$-semi-stable then so is $\V^{\sigma}$ as representation of $Q^{\sigma}$. Let $F \in \SI(Q,\beta)_{l \sigma}$ be a semi-invariant of weight $l \sigma$ such that $F(\V)\neq 0$. According to Proposition \ref{blow-ups-semi-inv-prop}, we can write $F=\varphi^{\#}(f)$ for some $f \in \SI(Q^{\sigma}, \beta)_{l \sigma}$, and so $F(\V)=f(\varphi(\V))=f(\V^{\sigma})\neq 0$. This shows that $\V^{\sigma}$ is $\sigma$-semi-stable.

To see why the last part of the theorem holds, recall that
\begin{itemize}
\item $(\V, \sigma)$ is an $ERP$ datum when
$$
\sigma \in \s_Q(\V) \Longleftrightarrow \capa_Q(\V, \sigma)>0
$$

\item $\sigma$ is $\V$-saturated when
$$
\sigma \in \s_Q(\V) \Longleftrightarrow n\sigma \in \s_Q(\V) \text{~for some integer~} n \geq 1.
$$
\end{itemize}
With the equivalence of $(i)$ and $(ii)$ at our disposal, it is now clear that 
$$
(\V, \sigma) \text{~is an ERP tuple} \Longleftrightarrow \sigma \text{~is~} \V-\text{saturated}.
$$

Finally, since membership to orbit semigroups for saturated weights is equivalent to the positivity of the capacity, Algorithm G in \cite[Corollary 3.17]{GarGurOliWig-2018} yields a deterministic polynomial time algorithm to check if $\sigma \in \s_Q(\V)$.
\end{proof}

\begin{remark} Let $Q$ be any Dynkin or Euclidean quiver and $\V \in \rep(Q, \beta)$ any representation. It has been proved in \cite{Chi-orbit-semigr-2009} that any weight $\sigma \in \ZZ^{Q_0}$ is $\V$-saturated. Consequently, there exists a deterministic polynomial time algorithm to check if $\Span_{\CC}({\V}^{i,j,p}_{q,r} \mid (i,j,p,q,r) \in \I_{\sigma})$ contains a non-singular matrix. 
\end{remark}

\section{Proof of Theorem \ref{main-thm-2}} The proof of Theorem \ref{main-thm-2} requires the following results about irreducible components of the form $\C(\alpha) \subseteq \module(A, \alpha)$ where $\alpha \in \NN^{Q_0}$ is a dimension vector with $\p_A(\alpha) \neq \emptyset$. This type of irreducible components also play an important role in \cite{BobZwa-2017}. 

\begin{lemma} \label{lemma-1-ext-0} Let $A=KQ/ \langle \R \rangle$ be an acyclic bound quiver algebra and $\alpha \in \NN^{Q_0}$ a dimension vector such that $\p_A(\alpha) \neq \emptyset$. Assume that $\C(\alpha)$ is an orbit closure and write $\C(\alpha)=\overline{\GL(\alpha)V}$ for some $V \in \p_A(\alpha)$. Then $\Ext^1_A(V,V)=0$, i.e.
$$
\ext^1_A(\C(\alpha), \C(\alpha))=0.
$$
\end{lemma}

\begin{proof} Since $\Ext^2_A(V,V)=0$ as $\pdim V \leq 1$, the following two facts hold.
\begin{enumerate}
\item Every infinitesimal deformation of $V \in \module(A, \alpha)$ can be lifted to a formal deformation, and thus the tangent space to $\module(A, \alpha)$ at $V$ coincides with the tangent space to the module scheme $\underline{\underline{\module}}(A, \alpha)$ (see \cite[Section 3.7]{Gei-1996}). Consequently, Voigt's isomorphism (see \cite[Chapter II]{Voi-77} or \cite[Section 2]{BobZwa-2017}) holds at the level of varieties as well, i.e. we have a natural isomorphism

$$
{T_V(\module(A, \alpha))\over T_V(\GL(\alpha)V)} \simeq \Ext^1_A(V,V).
$$

\item $V$ is a smooth point of $\module(A,\alpha)$ (see \cite[Section 3.7]{Gei-1996} or \cite[Sections 2.1 and 2.2]{delaPen-Sko-1996}), and so
$$
\dim T_V(\module(A, \alpha))=\dim \C(\alpha).
$$ 
\end{enumerate}

It now follows from $(1)$ and $(2)$ that $\dim \Ext^1_A(V,V)=0$.
\end{proof}

In what follows, we say that a non-zero $A$-module $M$ is \emph{$\theta$-stable} for an integral weight $\theta \in \ZZ^{Q_0}$ if $\theta \cdot \ddim M=0$ and $\theta \cdot \ddim M'<0$ for all proper submodules $0 \neq M' < M$. It is immediate to see that any $\theta$-stable module is a Schur module. 

\begin{lemma} \label{lemma-2-ext-0} Let $A=KQ/ \langle \R \rangle$ be an acyclic tame algebra, $\alpha \in \NN^{Q_0}$ a dimension vector, and $\theta=\langle \alpha, \cdot \rangle_A \in \ZZ^{Q_0}$. Assume that $\C(\alpha) \subseteq \module(A, \alpha)$ is an indecomposable irreducible component which is not an orbit closure. Then the generic module in $\C(\alpha)$ is $\theta$-stable, $\C(\alpha)$ contains $\Hom$-orthogonal modules, and
$$
\ext_A^1(\C(\alpha), \C(\alpha))=0.
$$
\end{lemma}

\begin{proof} Since $A$ is tame and $\C(\alpha)$ is an indecomposable irreducible component that is not an orbit closure, we can write
$$
\C(\alpha)=\overline{\cup_{\lambda \in \mathcal{U}} \GL(\alpha)V_{\lambda}},
$$
where $\{V_{\lambda}\}_{\lambda \in \mathcal{U}}$ is a $1$-parameter family of indecomposable $A$-modules with $\mathcal{U} \subseteq \mathbb{A}^1$ an open subset. On one hand, we know that $\p_A(\alpha)$ is an open subvariety of $\C(\alpha)$ by Proposition \ref{irr-comp-proj-1-prop}. On the other hand, we know from Crawley-Boevey's Homogeneity Theorem for tame algebras that all modules $V_{\lambda}$, with finitely many exceptions, are homogeneous, i.e. $V_{\lambda} \simeq \tau_A V_{\lambda}$ (see \cite{CB5}). Therefore, after possibly shrinking $\mathcal{U}$, we can assume that each $V_{\lambda}$ has projective dimension at most one, and $\tau_A V_{\lambda} \simeq V_{\lambda}$.\\ 

\noindent
\textbf{Claim 1:} $\C(\alpha)$ contains a Schur $A$-module.

\begin{proof}[Proof of \textbf{Claim 1}] From the discussion above and  \cite[Lemma 3]{CC13}, we can choose a $\lambda_0 \in \mathcal{U}$ such that
\begin{enumerate}[(i)]
\item $\pdim V_{\lambda_0} \leq 1$;
\item $\tau_A V_{\lambda_0} \simeq V_{\lambda_0}$;
\item $\dim \GL(\alpha)-\dim \C(\alpha)=\dim_K \End_A(V_{\lambda_0})-1$.
\end{enumerate}
As in the proof of Lemma \ref{lemma-1-ext-0}, using (i) we get that
$$
\dim \C(\alpha)-\dim \GL(\alpha) V_{\lambda_0}=\dim \Ext_A^1(V_{\lambda_0}, V_{ \lambda_0}).
$$
This combined with (iii) yields that $\dim \Ext_A^1(V_{\lambda_0}, V_{\lambda_0})=1$. Finally, using (i), (ii), and the Auslander-Reiten duality, we get that
$$
\dim \End_A(V_{\lambda_0})=\dim \Ext_A^1(V_{\lambda_0}, V_{\lambda_0})=1,
$$
i.e. $V_{\lambda_0}$ is a Schur $A$-module.
\end{proof} 

Since the Schur $A$-modules of a fixed dimension vector form an open subvariety of the corresponding module variety, Claim 1 above allows us, after possibly shrinking $\mathcal{U}$, to assume that each $V_{\lambda}$ is a Schur $A$-module, as well.\\

\noindent
\textbf{Claim 2:} $V_{\lambda}$ is a $\theta$-stable $A$-module for all $\lambda \in \mathcal{U}$.

\begin{proof}[Proof of \textbf{Claim 2}] The canonical weight $\theta^{V_{\lambda}} \in \ZZ^{Q_0}$ associated to $V_{\lambda}$ (see \cite{AdaIyaRei-2014, Domo}) is defined as follows
$$
\theta^{V_{\lambda}}(x)=\dim_K \Hom_A(V_{\lambda}, S_x)-\dim_K \Hom_A(S_x, \tau_A V_{\lambda}), \forall x \in Q_0.
$$
Then, since $V_{\lambda}$ is a homogeneous Schur $A$-module, $V_{\lambda}$ is stable with respect to $\theta^{V_{\lambda}}$ for all $\lambda \in \mathcal{U}$ (see \cite[Lemma 5]{ChiKinWey-2015}). Moreover, we have that $\theta^{V_{\lambda}}=\langle \alpha, \cdot \rangle_A$ since $\pdim V_{\lambda} \leq 1$. Thus all modules $V_{\lambda}$ are $\theta$-stable.
\end{proof} 

Now, let us choose any two distinct scalars $\lambda, \lambda' \in \mathcal{U}$. Then, according to Claim 2, $V_{\lambda}$ and $V_{\lambda'}$ are non-isomorphic, homogeneous $\theta$-stable modules. Being non-isomorphic and $\theta$-stable yields
$$
\Hom_A(V_{\lambda}, V_{\lambda'})=\Hom_A(V_{\lambda'}, V_{\lambda})=0.
$$

This Hom-orthogonality combined with the Auslander-Reiten duality and the fact that the two modules are homogeneous gives
$$
\Ext^1_A(V_{\lambda}, V_{\lambda'})=\Ext^1_A(V_{\lambda'}, V_{\lambda})=0.
$$
So, we get that $\ext_A^1(\C(\alpha), \C(\alpha))=0$.
\end{proof}

\begin{remark} Let $\C(\alpha)$ be an irreducible component as in Lemma \ref{lemma-2-ext-0}. Then, for a generic $M \in \C(\alpha)$ with $\pdim_A M \leq 1$ and $M\simeq \tau_AM$, Voigt's isomorphism and the Auslander-Reiten duality yield
$$
\codim_{\C(\alpha)}(\GL(\alpha)M)=\dim \Hom_A(M, \tau_A M),
$$
so $\C(\alpha)$ is generically $\tau$-reduced in the sense of \cite{GeiLabSch-2020}. Thus $\mathbf{Claim 1}$ in the proof above also follows from \cite[Theorem 3.2(ii)]{GeiLabSch-2020}. Nonetheless, we have provided the short, simple proof of $\mathbf{Claim 1}$ for completeness.
\end{remark}

We are now ready to prove Proposition \ref{main-prop}.

\begin{proof}[Proof of Proposition \ref{main-prop}] Part (i) follows from Lemmas \ref{lemma-1-ext-0} and \ref{lemma-2-ext-0}. For part (ii), let 
$$
\C(\alpha)=\overline{\C(\alpha_1) \oplus \ldots \oplus \C(\alpha_l)}
$$
be the generic decomposition of $\C(\alpha)$ with $\C(\alpha_1)$, $\ldots$, $\C(\alpha_l)$ indecomposable irreducible components such that $\ext^1_A(\C(\alpha_i), \C(\alpha_j))=0$ for all $1 \leq i \neq j \leq l$ (see Theorem \ref{decomp-irr-thm}{(1)}). We also know from part $(1)$ that $\ext^1_A(\C(\alpha_i), \C(\alpha_i))=0$ for all $1 \leq i \leq l$ . Using Theorem \ref{decomp-irr-thm}{(2)}, we get that
$$
\overline{\C(\alpha_1)^{\oplus n} \oplus \ldots \oplus \C(\alpha_l)^{\oplus n}}
$$
is an irreducible component of $\module(A, n \alpha)$. Moreover, since the generic module of this irreducible component has projective dimension at most one, we must have that 
$$
\C(n \alpha)=\overline{\C(\alpha_1)^{\oplus n} \oplus \ldots \oplus \C(\alpha_l)^{\oplus n}}.
$$

For part (iii), consider the generic decomposition of $\C(\alpha)$ 
$$
\C(\alpha)=\overline{\C(\alpha_1) \oplus \ldots \oplus \C(\alpha_l)}
$$
where $\C(\alpha_1)$, $\ldots$, $\C(\alpha_l)$ are indecomposable irreducible components such that $\ext^1_A(\C(\alpha_i), \C(\alpha_j))=0$ for all $1 \leq i \neq j \leq l$. We know from \cite[Section 2.3]{GeiSch-2005} that
$$
\mu_g(\C(\alpha))=\sum_{i=1}^l \mu_g(\C(\alpha_i)).
$$
(Recall that for an affine $G$-variety $X$, where $G$ is a linear algebraic group, the number of generic parameters is $\mu_g(X):=\dim X-\max \{\dim Gx \mid x \in X\}.$) It follows from Lemmas \ref{lemma-1-ext-0} and \ref{lemma-2-ext-0} that
$$\mu_g(\C(\alpha_i))=
\begin{cases}
1 \text{~if~} \C(\alpha) \text{~is not an orbit closure}\\
0 \text{~if~} \C(\alpha_i) \text{~is an orbit closure}
\end{cases} 
$$
So we get that
\begin{equation} \label{gen-number-param-eqn-1}
\mu_g(\C(\alpha))=t,
\end{equation} where $t$ is the number of indecomposable irreducible components (counting multiplicities) in the generic decomposition of $\C(\alpha)$ which are not orbit closures.

On the other hand, if $X=\bigoplus_{i=1}^l X_i$ is a generic representation in $\C(\alpha)$ then 
\begin{equation} \label{gen-number-param-eqn-2}
\mu_g(\C(\alpha))=\dim \C(\alpha)-\dim \GL(\alpha)+\dim_K \End_A(X)\geq \dim \C(\alpha)-\dim \GL(\alpha)+l
\end{equation}

\noindent
$(\Longrightarrow)$ It follows from $(\ref{gen-number-param-eqn-1})$ and $(\ref{gen-number-param-eqn-2})$ that if $\dim \C(\alpha)=\dim \GL(\alpha)$ then $t=l$. \\

\noindent
$(\Longleftarrow)$ For this implication, Theorem \ref{decomp-irr-thm}{(2)} and Lemma \ref{lemma-2-ext-0} allow us to choose the $X_i \in \C(\alpha_i)$, $i \in \{1, \ldots, l\}$, such that they are all homogeneous, Schur $A$-modules of projective dimension at most one, and $\Ext_A(X_i, X_j)=0$ for all $1 \leq i \neq j \leq l$. Then $\Hom_A(X_i, X_j)=0$ by the Auslander-Reiten duality, and so we get that
$$
\dim_K \End_A(X)=\sum_{i=1}^l \dim_K \End_A(X_i)=l.
$$

Using $(\ref{gen-number-param-eqn-1})$ and $(\ref{gen-number-param-eqn-2})$, we finally obtain that $l=\mu_g(\C(\alpha))=\dim \C(\alpha)-\dim \GL(\alpha)+l$, i.e.
$$
\dim \C(\alpha)=\dim \GL(\alpha).
$$

\end{proof}

Finally, we are now in a position to prove Theorem \ref{main-thm-2}.

\begin{proof}[Proof of Theorem \ref{main-thm-2}] Let $\sigma \in \ZZ^{Q_0}$ be a weight in $\LL_{\V}$ and $n \geq 1$ an integer such that $n \sigma \in \s_Q(\V)$. So, we can write $\restr{\sigma}{\supp(\beta)}=\langle \alpha, -\rangle_{B}$ with $\alpha \in \NN^{\supp(\beta)}$ such that $\p_B(\alpha) \neq \emptyset$. Let 
$$
\C(\alpha)=\overline{\C(\alpha_1) \oplus \ldots \oplus \C(\alpha_l)}
$$
be the generic decomposition of $\C(\alpha) \subseteq \module(B, \alpha)$ with $\C(\alpha_1)$, $\ldots$, $\C(\alpha_l)$ indecomposable irreducible components. By Proposition \ref{main-prop}, we have that
$$
\C(n \alpha)=\overline{\C(\alpha_1)^{\oplus n} \oplus \ldots \oplus \C(\alpha_l)^{\oplus n}}.
$$

Now we know from Proposition \ref{orbit-semigrp-homological-prop} that that the generic representation of $\C(n \alpha)$ is orthogonal to $\V$ since $n \sigma \in \s(\V)$. Consequently, there are representation $V^i_1, \ldots V^i_n \in \p_{B}(\alpha_i)$, $1 \leq i \leq l$, such that $\bigoplus_{i=1}^l \bigoplus_{j=1}^n V^i_j$ is orthogonal to $\V$. In particular, $V:=V^1_1\oplus V^2_1 \oplus \ldots \oplus V^l_1$ is an $\alpha$-dimensional $B$-module of projective dimension at most one such that $\Hom_{B}(V, \V)=\Ext^1_{B}(V, \V)=0$. Thus $\sigma \in \s(\V)$ by Proposition \ref{orbit-semigrp-homological-prop}. 
\end{proof} 

\begin{example} \label{examples} $(1)$ [\textbf{Tame examples}] For any acyclic tame algebra $A$ and $A$-module $\V \in \module(A, \beta)$, if $\sigma \in L_\V$ then $\sigma$ is $\V$-saturated by Theorem \ref{main-thm-2}.\\

\noindent
$(2)$ [\textbf{Wild Schur-tame examples}] Let us consider the bound quiver algebra $A=KQ/I$ where
$$
Q=\vcenter
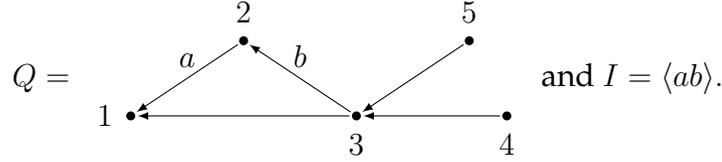
{\hbox{  
\begin{tikzpicture}[point/.style={shape=circle, fill=black, scale=.3pt,outer sep=3pt},>=latex]
   \node[point,label={left:$1$}] (1) at (0,0) {};
   \node[point,label={above:$2$}] (2) at (1.5,1) {};
   \node[point,label={below:$3$}] (3) at (3,0) {};
   \node[point,label={above:$5$}] (5) at (4.5,1) {};
   \node[point,label={below:$4$}] (4) at (5,0) {};   
  \path[->]
   (2) edge  node[midway, above] {$a$} (1)
   (3) edge  node[midway, above] {$b$} (2)   
   (3) edge  (1)
   (5) edge  (3)
   (4) edge  (3);
   \end{tikzpicture} 
}} \text{~~~and~~~~} I=\langle a b \rangle. 
$$
Then $A$ is a wild Schur-tame algebra and, for every Schur $A$-module $\V$, either $\V(a)=0$ or $\V(b)=0$ (see \cite{CC13}). So, the algebra $A/\Ann_A(\V)$ is a quotient of the path algebra of a $\widetilde{\mathbb{D}}_4$ or $\mathbb{D}_5$ quiver; either way, $A/\Ann_A(\V)$ is a tame algebra. Thus, for this wild algebra, all the weights in the weight semigroup of any Schur module $\V$ are $\V$-saturated.  
\end{example}

\begin{remark} \label{final-rmk} Let $A=KQ/ \langle \R \rangle$ be an acyclic bound quiver algebra, $\V \in \module(A, \beta)$ an $A$-module, and $\sigma \in \ZZ^{Q_0}$ a weight. If $\sigma$ is not $\V$-saturated (equivalently, if $(\V, \sigma)$ is not an ERP datum) then the answer to Edmonds' problem is \emph{always} NO. On the  other hand, if $\sigma$ is $\V$-saturated, as we have seen in Theorem \ref{membership-thm-1}, there exists a deterministic polynomial time algorithm to decide whether the answer to Edmonds' problem for $(\V, \sigma)$ is YES or NO. So, Edmonds' problem comes down to deciding whether a given weight is $\V$-saturated or not. 

Let us now assume that $B:=A/\Ann_A(\V)$ is tame; for example, this always happen when $A$ is tame or if $A$ and $\V$ are as in Example \ref{examples}{(2)}. Then Edmonds' problem can be further reduced to the problem about deciding whether $\sigma$ lies in $\LL_{\V}$ or not. Indeed, if $\sigma \notin \LL_{\V}$ then the answer to Edmonds' problem for $(\V, \sigma)$ is NO since $\s_Q(\V) \subseteq \LL_{\V}$. Otherwise, Theorems \ref{membership-thm-1} and \ref{main-thm-2} tell us that there exists a deterministic polynomial time algorithm to decide whether  $\Span_{\CC}(W^{i,j,p}_{q,r} \mid (i,j,p,q,r) \in \I_{\sigma})$ contains a non-singular matrix.

We are thus led to the following general problem: 
\begin{itemize}
\item[] \emph{Given an acyclic bound quiver algebra $\Lambda=KQ/ \langle \R \rangle$ and a dimension vector $\alpha \in \NN^{Q_0}$, decide whether $\p_{\Lambda}(\alpha) \neq \emptyset$}. 
\end{itemize}
\noindent
We plan to address the algebraic complexity of this problem in a sequel to this work. 
\end{remark}


\end{document}